\documentclass[11pt]{amsart}
\numberwithin{equation}{section}
\numberwithin{figure}{section}
\numberwithin{table}{section}

\usepackage{macros}

\title{A Conjecture of Warnaar-Zudilin from Deformations of Lie Superalgebras}
\author{Thomas Creutzig$^1$}
\address{${}^1$Department of Mathematics\\ FAU Erlangen\\ 91058 Erlangen, Germany}
\email[T.~Creutzig]{creutzig@ualberta.ca}

\author{Niklas Garner$^2$}
\address{${}^2$Mathematical Institute\\ University of Oxford\\ Oxford, OX2 6GG, UK}
\email[N.~Garner]{niklas.garner@maths.ox.ac.uk}

\begin{document}
	
\begin{abstract}
	We prove a collection of $q$-series identities conjectured by Warnaar and Zudilin and appearing in recent work with H. Kim in the context of superconformal field theory.
	Our proof utilizes a deformation of the simple affine vertex operator superalgebra $L_k(\mathfrak{osp}_{1|2n})$ into the principal subsuperspace of $L_k(\mathfrak{sl}_{1|2n+1})$ in a manner analogous to earlier work of Feigin-Stoyanovsky.
	This result fills a gap left by Stoyanovsky, showing that for all positive integers $N$, $k$ the character of the principal subspace of type $A_N$ at level $k$ can be identified with the (super)character of a simple affine vertex operator (super)algebra at the same level.
\end{abstract}

\maketitle
\tableofcontents

\section{Introduction}

Vertex operator algebras (VOAs) appear in a wide variety of physical problems and are a rich source of mathematics that serve as a fertile bridge between algebra and number theory.
In this paper, we capitalize on insight gleaned from a recent physical analysis \cite{CGK} with H. Kim to establish a family of $q$-series identities of interest to number theorists.
\begin{thm}[Main result, Corollary \ref{cor:characters}]
	For every positive integer $n, k$ the following is an equality of $q$-series: 
	\begin{equation}
	\label{eq:charactersintro}
	\frac{1}{(q)_\infty^{n(2n-1)}}\sum_{u \in \Z^n}(-1)^{|u|} \xi(u) q^{(k+n+\scriptstyle{\frac{1}{2}})|\!|u|\!|^2+ \rho \cdot u} = \sum_{m \in \mathbb{N}^{k(2n-1)}} \frac{q^{\scriptstyle{\frac{1}{2}}\sum\limits_{i,j=1}^{k}\sum\limits_{a,b=1}^{2n-1} T_{ij} A_{ab} m_{ia}m_{jb}}}{\prod\limits_{i=1}^k \prod\limits_{a=1}^{2n-1} (q)_{m_{ia}}}
	\end{equation}
\end{thm}
In this theorem, the matrices $A_{ab}$ and $T_{ij}$ are simply the Cartan matrix of $A_{2n-1}$ and the inverse of the Cartan matrix of the rank $k$ tadpole graph
\[
A_{ab} = \begin{cases}
	2 & a = b\\
	-1 & a = b \pm 1\\
	0 & \text{else}
\end{cases} \qquad \text{and} \qquad T_{ij} = \min(i,j)\,,
\]
$|u|$ denotes the sum of the components of $u$, $|\!|u|\!|^2 = u \cdot u$ the squared norm of $u$, $\rho$ the vector with components $\rho_m = m - \frac{1}{2}$, and finally
\[
\xi(u) = \prod_{1 \leq l < m \leq n} \frac{v_m^2 - v_l^2}{\rho_m^2 - \rho_l^2} \qquad v_m = \rho_m + (2(n+k)+1)u_m\,.
\]
This collection of $q$-series identities was proved for $k = 1$ and conjectured to hold more generally by Warnaar and Zudilin \cite[Conjecture 1.1, Theorem 1.2]{WarnaarZudilin} and were viewed by these authors as an extension of the Rogers-Ramanujan \cite{Rogers, Andrews84} (for $n = k = 1$) and Andrews-Gordon \cite{Gordon61,Andrews74} (for $n = 1$) identities and a specialization of the $A^{(2)}_{2n}$ Macdonald identities \cite{Macdonald71} (for $k = 0$, where the right-hand side reduces to $1$).%
\footnote{We note that the parameter $k$ in \cite{WarnaarZudilin} differs from our $k$ by 1: $k_{WZ} = k + 1$.} %

We instead view the $q$-series appearing on the two sides of the Eq. \eqref{eq:charactersintro} as supercharacters of VOAs.
The left-hand side of Eq. \eqref{eq:charactersintro} is the supercharacter of the simple affine VOA $L_k(\fosp_{1|2n})$; see Appendix \ref{app:osp} for more details.
The right-hand side of Eq. \eqref{eq:charactersintro} is the character of a different VOA: as conjectured (and proven for $\fsl_2$) by Feigin-Stoyanovsky \cite{StoyanovskyFeigin} and proven more generally by Georgiev \cite{Georgiev}, it is the character of the principal subspace of $L_k(\fsl_{2n})$; we build on these results to show that it is also the supercharacter of the principal subspace of $L_k(\fsl_{1|2n+1})$.
An analogous family of $q$-series identities appeared in work of Stoyanovsky \cite{Stoyanovsky1998} (where the left-hand side is replaced by the character of the simple affine VOA $L_k(\fsp_{2n})$ and the right-hand side is replaced by the character of the principal subspace of $L_k(\fsl_{2n+1})$) by showing that $\fsp_{2n}$ can be deformed into a nilpotent subalgebra of $\fsl_{2n+1}$, that this deformation can be extended to the affine setting, and that its character does not change under this deformation.
We establish Eq. \eqref{eq:charactersintro} by similarly deforming $\fosp_{1|2n}$ into a nilpotent subalgebra of $\fsl_{1|2n+1}$.
The equality of the $q$-series in Eq. \eqref{eq:charactersintro} is thus an immediate consequence of the following theorem:
\begin{thm}[Theorem \ref{thm:characters}]
	The supercharacter of $L_k(\fosp_{1|2n})$ is equal to both the supercharacter of the principal subspace of $L_k(\fsl_{1|2n+1})$ and to the character of the principal subspace of $L_k(\fsl_{2n})$.
\end{thm}

Our result fills a gap left by Stoyanovsky, i.e. that the vacuum character of the principal subspace of $L_k(\fsl_{N})$ can be identified with the vacuum character of a simple affine VOA $L_k(\fg_N)$ for all integers $N \geq 2$ and $k \geq 1$, where $\fg_{2n+1} \simeq \fsp_{2n}$ and $\fg_{2n} \simeq \fosp_{1|2n}$.
We also note that an immediate corollary of this result establishes part of a conjecture of Bringmann, Calinescu, Folsom, and Kimport concerning the modularity of the right-hand side of Eq. \eqref{eq:charactersintro}.
Explicitly, together with their Theorem 1.2 this result proves Case i. of Conjecture 4.1 of \cite{BCFK} for $N$ even and $\ell = k$, cf. Remark 7 of loc. cit.

\subsection*{Acknowledgements}
We would like to thank Heeyeon Kim for collaboration inspiring this work. We would also like to thank S. Ole Warnaar, Shashank Kanade, and especially Christopher Sadowski for useful comments on a preliminary draft of this document.
NG is supported by the ERC Consolidator Grant No. 864828, titled “Algebraic Foundations of Supersymmetric Quantum Field Theory'' (SCFTAlg) and was previously supported by funds from the Department of Physics and the College of Arts \& Sciences at the University of Washington, Seattle.

\section{Deformations of Lie superalgebras}
\label{sec:finite}
In this section we study certain deformations of the simple Lie algebras $\fsp_{2n}$, $\fso_m$, and $\fosp_{m|2n}$ into principal nilpotent subalgebras of $\fsl_{2n+1}$, $\fsl_{m}$, and $\fsl_{m|2n+1}$, respectively.
The first two examples are meant to elaborate on \cite{Stoyanovsky1998}, and the last example is a simple generalization of the construction.

\subsection{$\fsp_{2n}$}
\label{sec:sp}
We start with the main example of \cite{Stoyanovsky1998}.
Consider the (odd) vector space $\C^{0|2n+1} = \Pi \C^{2n+1}$ and choose a basis $\theta_a$ for $a = 0, \dots, 2n$.
We introduce a (degenerate) supersymmetric bilinear form on $\C^{0|2n+1}[x] := \C^{0|2n+1} \otimes_\C \C[x]$ given by the (super)symmetric quadratic tensor
\[
	B = \sum_{a=0}^{2n-1} x^a \theta_a \theta_{a+1}\,.
\]
We then consider the family of Lie algebras $\fg$ defined as the subalgebra of $\fsl_{0|2n+1}[x] \simeq \fsl_{2n+1}[x]$ that preserves $\theta_0$ and $B$; set $\fg_\epsilon := \fg \otimes_{\C[x]} \C[x]/(x - \epsilon)$.
We aim to show that for $\epsilon \neq 0$ this Lie algebra can be identified with $\fsp_{2n}$ and that for $\epsilon = 0$ the subalgebra tends to the nilpotent subalgebra of $\fsl_{2n+1}$.

\begin{eg}
	For $n = 1$ the algebra $\fg$ is spanned (as a $\C[x]$-module) by
	\[
		h_x = \begin{pmatrix}
			0 & 0 & 1\\
			0 & x & 0\\
			0 & 0 & -x
		\end{pmatrix} \qquad e_x = \begin{pmatrix}
			0 & 0 & 0\\
			0 & 0 & -1\\
			0 & 0 & 0\\
		\end{pmatrix} \qquad f_x = \begin{pmatrix}
			0 & 1 & 0\\
			0 & 0 & 0\\
			0 & -x & 0\\
		\end{pmatrix}
	\]
	which satisfy the following commutation relations:
	\[
		[h_x, e_x] = 2 x e_x \qquad [h_x, f_x] = - 2 x f_x \qquad [e_x, f_x] = h_x
	\]
	It is evident from these commutators that for non-zero $\epsilon$ the algebra $\fg_\epsilon$ is precisely $\fsp_2$.
	Moreover, $\fg_0$ is the strictly upper-triangular subalgebra.
\end{eg}

\begin{lem}
	$\fg_0$ is isomorphic to the strictly upper triangular subalgebra of $\fsl_{2n+1}$.
\end{lem}

\begin{proof}
	We make explicit the constraints imposed by invariance of $\theta_0$ and $B$ under $X \in \fsl_{2n+1}[x]$.
	Let $X^a{}_b \in \C[x]$ be the matrix elements of $X$ in the basis $\theta_a$.
	Preserving $\theta_0$ implies $X^a{}_0 = 0$ and preserving $B_x$ implies
	\begin{equation}
	\label{eq:spconstraint}
		x^{b-1} X^a{}_{b-1} - x^b X^a{}_{b+1} + x^a X^b{}_{a+1} - x^{a-1} X^b{}_{a-1} = 0
	\end{equation}
	for all $a, b = 0, \dots 2n$, where we set $X^a{}_{-1} = 0$ and $X^a{}_{2n+1} = 0$; it suffices to consider $a < b$ due to the anti-symmetry of this constraint.
	We will deduce the lemma by showing that these constraints imply $X^a{}_b$ is unconstrained for $a < b$ and belongs to $x \C[x]$ otherwise.
	
	We proceed by induction.
	Consider first the case $a = 0$, where Eq. \eqref{eq:spconstraint} implies
	\[
		X^b{}_1 = x^{b-1}\big(x X^0{}_{b+1} - X^0{}_{b-1}\big)
	\]
	for all $1 \leq b \leq 2n$.
	This is manifestly proportional to $x$ for $b > 1$; this also holds for $b = 1$ because $X^0{}_0 = 0$.
	In particular, $X^b{}_1$ is constrained to be a (linear) polynomial of the strictly upper-triangular matrix elements for all $b$ and is proportional to $x^{b-1}$ for $b > 1$ and proportional to $x$ for $b = 1$.
	
	Now suppose for all $a \leq a_0$ and $b \geq a$ we have shown that the matrix element $X^b{}_a$ is constrained to be a (linear) polynomial in the strictly upper-triangular matrix elements that is proportional to $x^{b-a}$ for all $b > a$ and proportional to $x$ for $b = a$.
	Rearranging Eq. \eqref{eq:spconstraint} for $a = a_0$ gives
	\[
		X^b{}_{a_0+1} = x^{-1} X^b{}_{a_0-1} + x^{b-a_0-1}\big(x X^{a_0}{}_{b+1} - X^{a_0}{}_{b-1}\big)\,.
	\]
	The inductive hypothesis implies that the right-hand side is necessarily a (linear) polynomial in the strictly upper-triangular matrix elements.
	Moreover, the first term is proportional to $x^{b-a_0}$ and the second term is proportional to $x^{b-a_0-1}$ for all $b > a_0+1$ and proportional to $x$ for $b = a_0+1$, as desired.
\end{proof}

\begin{rmk}
	We make the solution to the above constraints explicit for $n = 2$.
	A general element $X$ takes the form
	\[
	\begin{pmatrix}
		0 & X^0{}_1 & X^0{}_2 & X^0{}_3 & X^0{}_4\\
		0 & x X^0{}_2 & X^1{}_2 & X^1{}_3 & X^1{}_4\\
		0 & -x X^0{}_1 + x^2 X^0{}_3 & -x X^0{}_2 + x X^1{}_3 & X^2{}_3 & X^2{}_4\\
		0 & -x^2 X^0{}_2 + x^3 X^0{}_4 & -x X^1{}_2 + x^2 X^1{}_4 & -x X^1{}_3 + x X^2{}_4 +x^2 X^0{}_4 & X^3{}_4\\
		0 & -x^3 X^0{}_3 & -x^2 X^1{}_3 & -x X^2{}_3 - x^2 X^0{}_3 & - x X^2{}_4 - x^2 X^0{}_4\\
	\end{pmatrix}
	\]
\end{rmk}

\begin{lem}
	\label{lem:sp2n}
	$\fg_\epsilon$ can be identified with $\fsp_{2n}$ when $\epsilon \neq 0$.
\end{lem}

\begin{proof}
	We consider the following elements of $\C^{0|2n+1}\otimes_\C[x]$:
	\[
		\theta = \theta_0 \qquad \phi_{2i-1} = x^{i-1} \theta_{2i-1} \qquad \phi_{2i} = x^{i-1}(x \theta_{2i} - \theta_{2i-2})
	\]
	where $i = 1, \dots, n$; these elements are conjugate to the basis $\theta_a$ if we localize $x$, i.e. they are conjugate bases of $\C^{0|2n+1}[x^{\pm 1}]$ as a $\C[x^{\pm1}]$-module.
	In this basis $B$ takes the form
	\[
		B = \sum_{i=1}^{n} \phi_{2i-1} \phi_{2i}
	\]
	which is non-degenerate when restricted to the subspace spanned by the $\phi$'s.
	In particular, we conclude that the subalgebra of $\fsl_{2n+1}[x^{\pm 1}]$ preserving $\theta = \theta_0$ and $B$ is precisely the $\fsp_{2n}[x^{\pm 1}]$ rotating the $\phi$'s.
\end{proof}

\begin{rmk}
	The original basis $\theta_a$ is not a weight basis for the Cartan subalgebra of $\fsp_{2n}$ rotating the pairs $\phi_{2i-1}, \phi_{2i}$ with opposite weights and preserving $\theta$.
	Both are eigenbases for the linear transformation that scales $\theta_a$ for $a$ odd (resp. even) with weight $1$ (resp. $-1$); this transformation preserves $B$, but fails to preserve $\theta_0$ (only preserving its span) and does not belong to $\fsl_{2n+1}$ as it has trace $1$.
	Nonetheless, it induces a $2\Z$-grading on the family $\fg_\epsilon$.
	For $n = 1$ and $\epsilon \neq 0$ it can be identified with the $h_x$ weight grading.
\end{rmk}

\subsection{$\fso_{m}$}
\label{sec:so}

We now describe the example mentioned in Remark 2 of \cite{Stoyanovsky1998}, suggested to the author by Panov.
We consider the (even) vector space $\C^{m|0} = \C^m$ and choose a basis $z_r$, $r = 1, \dots, m$.
We then consider the symmetric quadratic tensor
\[
	C = \tfrac{1}{2}\sum_{r=1}^{m} x^{2(r-1)} z_{r}^2\,.
\]
We define the Lie algebra $\ff$ as the subalgebra of $\fsl_{m|0}[x] \simeq \fsl_m[x]$ preserving $C$ and denote $\ff_\epsilon := \ff \otimes_{\C[x]} \C[x]/(x - \epsilon)$.
The analog of Eq. \eqref{eq:spconstraint} for $Y = (Y^r{}_s) \in \fsl_m[x]$ takes the form
\begin{equation}
\label{eq:soconstraint}
	x^{2r} Y^s{}_r + x^{2s} Y^r{}_s = 0
\end{equation}
which implies $Y^s{}_r = x^{2(s-r)} Y^r{}_s$ for all $s \geq r$.
We can similarly consider the conjugate basis $w_r = x^{r-1} z_r$, where the quadratic tensor takes the form
\[
	C = \tfrac{1}{2}\sum_{r=1}^{m} w_{r}^2\,.
\]

\begin{lem}
	$\ff_\epsilon$ is isomorphic to the strictly upper-triangular subalgebra of $\fsl_{m}$ for $\epsilon = 0$ and isomorphic to $\fso_{m}$ for $\epsilon \neq 0$.
\end{lem}

\subsection{$\fosp_{m|2n}$}
\label{sec:osp}
We now combine the above examples and consider the (super) vector space $\C^{m|2n+1} = \C^m \oplus \Pi \C^{2n+1}$ with (homogeneous) basis $z_r, \theta_a$.
We then consider the (super)symmetric quadratic tensor
\[
	D = \tfrac{1}{2}\sum_{r=1}^m x^{2(r-1)} z^2_{r} + x^{2m} \sum_{a=0}^{2n-1} x^{2a} \theta_a \theta_{a+1}
\]
We define the Lie superalgebra $\fk$ as the subalgebra of $\fsl_{m|2n+1}[x]$ preserving $\theta_0$ and $D$ and define $\fk_\epsilon$ as before. 

\begin{eg}
	For the case $m = n = 1$ we find that $\fk$ has even generators
	\[
		h_x = \left(\begin{array}{c | c c c}
			0 & 0 & 0 & 0\\ \hline
			0 & 0 & 0 & 1\\
			0 & 0 & x^2 & 0\\
			0 & 0 & 0 & -x^2\\
		\end{array}\right) \qquad e_x = \left(\begin{array}{c | c c c}
			0 & 0 & 0 & 0\\ \hline
			0 & 0 & 0 & 0\\
			0 & 0 & 0 & -1\\
			0 & 0 & 0 & 0\\
		\end{array}\right) \qquad f_x = \left(\begin{array}{c | c c c}
			0 & 0 & 0 & 0\\ \hline
			0 & 0 & 1 & 0\\
			0 & 0 & 0 & 0\\
			0 & 0 & -x^2 & 0\\
		\end{array}\right)
	\]
	and odd generators
	\[
		\psi_{1,x} = \left(\begin{array}{c | c c c}
			0 & 0 & 0 & 1\\ \hline
			0 & 0 & 0 & 0\\
			x^4 & 0 & 0 & 0\\
			0 & 0 & 0 & 0\\
		\end{array}\right) \qquad \psi_{2,x} = \left(\begin{array}{c | c c c}
			0 & 0 & 1 & 0\\ \hline
			x^2 & 0 & 0 & 0\\
			0 & 0 & 0 & 0\\
			-x^4 & 0 & 0 & 0\\
		\end{array}\right)
	\]
\end{eg}

\begin{lem}
	\label{lem:ospm2n}
	The subalgebra $\fk_\epsilon$ is isomorphic to the strictly upper-triangular subalgebra of $\fsl_{m|2n+1}$ preserving $\theta_0$ for $\epsilon = 0$ and is isomorphic to $\fosp_{m|2n}$ when $\epsilon \neq 0$.
\end{lem}

\begin{rmk}
	As is evident in the above example, the subalgebra $\fk_\epsilon$ does not deform to the entire strictly upper-triangular subalgebra of $\fsl_{m|2n+1}$ unless $m = 0$.
	Indeed, the odd subspace of the latter has dimension $m(2n+1)$ whereas the odd subspace of $\fosp_{m|2n}$ is dimension $2mn$.
\end{rmk}

\begin{proof}
	The second assertion follows from the changes of basis introduced in Sections \ref{sec:sp} and \ref{sec:so}.
	
	To establish the first assertion, we note that the even subalgebra of $\fk$ is simply $\ff \oplus \fg$ --- the additional even generator on top of $\fsl_m[x] \oplus \fsl_{2n+1}[x]$ does not preserve $\theta_0$.
	We see that the even subalgebra of $\fk_0$ is simply the nilpotent subalgebra of $\fsl_m \oplus \fsl_{2n+1}$.
	To check the odd subspace of $\fk_0$, we consider (odd) linear map sending $z_r \to U^a{}_r \theta_a$ and $\theta_a \to V^r{}_a z_r$, with $V^r{}_0 = 0$; we must verify that $U^a{}_r \to 0$ is proportional to $x$.
	Preserving $D$ requires
	\[
		U^a{}_r = x^{2(m-r+a)} \big(x^2 V^r{}_{a+1} - V^r{}_{a-1}\big)
	\]
	for $a = 0, \dots, 2n$ and $r = 1, \dots, m$, where we set $V^r{}_{-1} = 0$ and $V^r{}_{2n+1} = 0$.
	In particular, we must that have $U^a{}_r$ is proportional to $x$ for all $a$, $r$.
\end{proof}

\section{Duflo-Serganova reduction of deformations}
\label{sec:DufloSerganova}

The following is a special case of cohomological reduction \cite{Candu:2010yg} as well as the Duflo-Serganova functor \cite{Serganova:2011zz, gorelik2017}.

\subsection{$\fgl_{1|1}$ preliminaries}

We fix a basis $N, E, \psi^\pm$ of $\fgl_{1|1}$ with even generators $N, E$, odd generators $\psi^\pm$, and commutation relations
\[
	[N, \psi^\pm] = \pm  \psi^\pm, \qquad [\psi^+, \psi^-] = E
\]
and  $E$ the central element of $\fgl_{1|1}$.
We call a $\fgl_{1|1}$-module a weight module if its Cartan subalgebra acts semisimply.
As it is central, $E$ acts on an indecomposable module by multiplication by a scalar and so the category of weight modules admits a block decompositon
\[
	\CC = \bigoplus_{e \in \C} \CC_e
\]
with $\CC_e$ the block on which $E$ acts by the scalar $e$. 
Let $\CD_0$ be the category whose objects are direct sums of trivial representations of $\fgl_{1|1}$ and set
\[
	\CD := \CD_0 \oplus \bigoplus_{e \in \R_{>0}} \CC_e.
\]
This category is semisimple and closed under tensor products.
We complete this category to allow for countable direct sums of objects and denote this completion by the symbol $\overline{\CD}$.

For $M$ an object of $\CD$, let $H(M)$ denote the cohomology of $M$ with respect to $\psi^+$.
\begin{lem}
	\label{lem:coho}
	Let $M$ be an object in $\overline{\CD}$ and let
	\[
		M = \bigoplus_{e \in \R_{\geq 0}} M_e
	\]
	be its decomposition into $E$-eigenspaces.
	Then $H(M) = M_0$, and if $M$ in $\CD$, then also $\text{sdim}(M) = \text{sdim}(M_0)$.
\end{lem}

\begin{proof}
	Let $e\neq 0 $, then a simple object $M_{n, e}$ in $\CC_e$ is two dimensional.
	It is generated by a highest-weight vector $v$, such that $Nv = nv$, $Ev=ev$ and $\psi^+v =0$.
	The module $M_{n,e}$ is spanned by $v, \psi^-v$ and hence has super dimension zero.
	Moreover, $\psi^+\psi^-v = [\psi^+, \psi^-]v = Ev = ev$ and so  $H(M_{n, e}) = 0$.
	
	For a general object $M$ of $\overline{\CD}$, because $\overline{\CD}$ is semisimple it follows that the submodule $M_e$ for $e\neq 0$ is a countable direct sum of $M_{n,e}$ (typically with different values of $n$ appearing) and hence $H(M_e) = 0$ and  $\text{sdim}(M_e) = 0$.
	As the action of $\psi^+$ is trivial on the submodule $M_0$ we conclude
	\[
		H(M) = \bigoplus_e H(M_e) = H(M_0) = M_0
	\]
	and additionally if $M$ in $\CD$, $\text{sdim}(M) = \text{sdim}(M_0)$.
\end{proof}

\subsection{Formalization of the deformation problem}

In this section we formalize our problem.
Let $\fg$ be a Lie superalgebra with non-degenerate bilinear form $\kappa$ and let $\fa$ be a subalgebra of $\fg \otimes_\C \C[x]$. Let $\fa_\epsilon = \fa \otimes_{\C[x]} \C[x]/(x-\epsilon)$ be a 1-parameter family of subalgebras; set $\fp := \fa_0$.
We denote by $\wh{\fg}$ the central extension
\[
	0 \to \C K \to \wh{\fg} \to \fg \otimes \C[t^{\pm 1}] \to 0
\]
induced by the bilinear form $\kappa$ and similarly denote $\wh{\fp} = \fp \otimes \C[t^{\pm 1}]$.
Let $V^k(\fg)$ be the universal affine VOA of $\fg$ at level $k\in \C$ associated to $\kappa$ and let $V^k(\fa_\epsilon)$, $V^k(\fp)$ be the affine subVOAs associated to $\fa_\epsilon$, $\fp$.
For $I$ an ideal in $V^k(\fg)$, we set $I_\epsilon = I \cap V^k(\fa_\epsilon)$ and $I_\fp := I_0$.
Denote the quotients by the corresponding ideals by $L_k(\fg), L_k(\fa_\epsilon), L_k(\fp)$.
There is thus a short exact sequence
\[
	0 \rightarrow I_\fp \rightarrow V^k(\fp) \rightarrow L_k(\fp) \rightarrow 0.
\]
Assume that there is an embedding $\iota: \fgl_{1|1} \hookrightarrow \fg$ such that $\fgl_{1|1}$ acts on $\fp$
\[
	[\iota(x), y] \in \fp \qquad \forall \ x \in \fgl_{1|1}, \ y \in \fp. 
\]
This action extends in the obvious way to $\wh\fp$.

\begin{lem}
	\label{lem:SES}
	With the above set-up, if $\fp$ is an object in $\CD$, then
	\[
		0 \rightarrow H(I_\fp) \rightarrow H(V^k(\fp)) \rightarrow H(L_k(\fp)) \rightarrow 0
	\]
	is an exact sequence of $H(V^k(\fp))$-modules.
\end{lem}

\begin{proof}
	$\iota$ extends to an embedding of $\fgl_{1|1}$ into $\wh\fg$ via the identification of $\fg$ with the horizontal subalgebra of $\wh\fg$.
	Since $\fp$ is an object in $\CD$, it follows that $\wh\fp$ is an object in $\overline{\CD}$.
	Since via this embedding any $\wh\fg$-module becomes a $\fgl_{1|1}$-module, we see that $I$ also belongs to $\overline{\CD}$.
	Moreover, $I_\fp = I \cap V^k(\fp)$ is a $\fgl_{1|1}$-module because $\fp$ is preserved by the $\fgl_{1|1}$ action and since $\overline{\CD}$ is closed under submodules it too is an object in $\overline{\CD}$.
	
	We now apply Lemma \ref{lem:coho}.
	The cohomology of $I_\fp, V^k(\fp)$ with respect to $\psi^+$ is just its subspace of $E$-weight zero, since $\fgl_{1|1}$ acts semisimply on $V^k(\fp)$ the same must be true for $L_k(\fp)$.
	The claim follows since the $E$-weight zero subspace of $V^k(\fp)$ is a subalgebra of $V^k(\fp)$.
\end{proof}

\begin{cor}
	\label{cor:supercharacters}
	With the above set-up, if $\fp$ is an object in $\CD$ and if every element in the $E$-weight zero subspace of $I_\fp$ deforms to an element in $I_\epsilon$, then supercharacters coincide,
	\[
		\text{sch}[H(L_k(\fp))] =  \text{sch}[L_k(\fp)] = \text{sch}[L_k(\fa_\epsilon)].
	\]
\end{cor}

\begin{proof}
	The first equality is a simple consequence of Lemma \ref{lem:coho} applied to the subspaces of $L_k(\fp)$ with a given conformal weight, which are objects in $\CD$.
	To see the second equality, it suffices to compare the supercharacters of the ideals $I_\fp = I_0$ and $I_\epsilon$ for $\epsilon \neq 0$ --- if the supercharacters of these ideals agree, then so too must the quotients because the character of $V^k(\fa_\epsilon)$ doesn't depend on $\epsilon$.
	
	To see that the supercharacter of $I_\epsilon$ doesn't depend on $\epsilon$, we note that the $\epsilon \to 0$ limit of every nonzero element of $I_\epsilon$ is a nonzero element of $I_0 = I_\fp$, although it may be that not all elements of $I_0$ can be obtained via such a limit.
	Nonetheless, as only the $E$-weight zero subspace of $I_0$ contributes to its supercharacter and we assume that every element thereof deforms, we can conclude that the supercharacters must agree.
\end{proof}

\section{Supercharacter formulae}
\label{sec:supercharacters}
In this section we affinize the deformation of $\fosp_{1|2n}$ into the principal subspace of $\fsl_{1|2n+1}$ and describe the resulting character identities, following \cite{Stoyanovsky1998}.
For this section we fix positive integers $n, k > 0$.

\subsection{Deforming the principal subspace of $\fsl_{1|2n+1}$}
Let $\fp$ denote the principal subalgebra of $\fg = \fsl_{1|2n+1}$.
We will work with the following basis for $\fsl_{1|2n+1}$.
We denote by $\Phi$ the set of roots of $\fsl_{2n+1}$ with respect to the Cartan subalgebra of diagonal matrices, $\Phi^+$ the set of positive roots corresponding to the upper nilpotent subalgebra, and $\Delta$ the set of simple roots.
The even generators will be denoted $E$, for the diagonal supermatrix with entries $(1,1,0,\dots, 0)$, $h_i$ the Cartan generators of $\fsl_{2n+1}$ corresponding to the simple roots $\alpha_i$, and $e_{\alpha}$, $f_\alpha$, for $\alpha \in \Phi^+$; the odd generators will be denoted $\psi_a$, $\overline{\psi}^a$, $a = 0, \dots, 2n$.
A basis for the subalgebra $\fp$ is given by the $e_\alpha$ for any $\alpha \in \Phi^+$ together with $\psi_a$ for $a = 1, \dots, 2n$.

\begin{rmk}
	As $\fp$ does not include the element $\psi_0$ corresponding to the odd simple root of $\fsl_{1|2n+1}$, the root corresponding to $\psi_1$ can be thought of as simple, i.e. it cannot be written as the sum of roots contained in the subalgebra $\fp$.
\end{rmk}

For an element $x \in \fg$ we denote by $x_j = x \otimes t^j$ the corresponding element of $\wh{\fg}$.
We will find the following formal series particularly useful: for any $\alpha \in \Phi^+$ we denote
\[
	e_\alpha(z) = \sum_{j \in \Z} e_{\alpha, j} z^{-j-1} \qquad f_\alpha(z) = \sum_{j \in \Z} f_{\alpha, j} z^{-j-1}
\]
and for any $a = 0, \dots 2n$ we denote
\[
	\psi_a(z) = \sum_{j \in \Z} \psi_{a, j} z^{-j-1} \qquad \overline{\psi}^a(z) = \sum_{j \in \Z} \overline{\psi}^a_{j} z^{-j-1}
\]
We also denote $e_i(z) = e_{\alpha_i}(z)$ and $f_i(z) = f_{\alpha_i}(z)$, $i = 1, \dots, 2n$.

Extending the deformation introduced in Section \ref{sec:osp} to the affine setting is straightforward.
We denote by $\wh{\fk}_\epsilon$ the central extension  of $\fk_{\epsilon} \otimes \C[t^{\pm1}]$ by $\C K$ obtained by restricting to $X \in \fk_\epsilon$.
The following result is an immediate consequence of Lemma \ref{lem:ospm2n}.

\begin{cor}
	\label{cor:affineosp}
	The Lie algebra $\wh{\fk}_\epsilon$ tends to $\wh{\fp}$ as $\epsilon \to 0$ and can be identified with $\wh{\fosp}_{1|2n}$ when $\epsilon \neq 0$.
\end{cor}

\begin{rmk}
	\label{rmk:gl11}
	We note that this example fits into the general setting of Section \ref{sec:DufloSerganova}.
	The embedding of $\fgl_{1|1}$ into $\fg = \fsl_{1|2n+1}$ is given the bosonic supermatrices
	\[
		N = \left(\begin{array}{c | c c c c}
			\tfrac{2n+1}{2n} & 0 & 0 & \dots & 0\\ \hline
			0 & \tfrac{1}{2n} & 0 & \dots & 0\\
			0 & 0 & \tfrac{1}{2n} & \dots & 0\\
			\vdots & \vdots & \vdots & \ddots & \vdots\\
			0 & 0 & 0 & \dots & \tfrac{1}{2n}\\
		\end{array}\right) \qquad E = \left(\begin{array}{c | c c c c}
			1 & 0 & 0 & \dots & 0\\ \hline
			0 & 1 & 0 & \dots & 0\\
			0 & 0 & 0 & \dots & 0\\
			\vdots & \vdots & \vdots & \ddots & \vdots\\
			0 & 0 & 0 & \dots & 0\\
		\end{array}\right)
	\]
	and the fermionic supermatrices
	\[
		\psi^+ = \left(\begin{array}{c | c c c c}
			0 & 1 & 0 & \dots & 0\\ \hline
			0 & 0 & 0 & \dots & 0\\
			0 & 0 & 0 & \dots & 0\\
			\vdots & \vdots & \vdots & \ddots & \vdots\\
			0 & 0 & 0 & \dots & 0\\
		\end{array}\right) \qquad \psi^- = \left(\begin{array}{c | c c c c}
			0 & 0 & 0 & \dots & 0\\ \hline
			1 & 0 & 0 & \dots & 0\\
			0 & 0 & 0 & \dots & 0\\
			\vdots & \vdots & \vdots & \ddots & \vdots\\
			0 & 0 & 0 & \dots & 0\\
		\end{array}\right)\,.
	\]
	It is clear that the principal subspace $\fp$ indeed belongs to $\CD$.
	We note that this copy of $\fgl_{1|1}$ does not preserve $\fk_\epsilon$ for general $\epsilon$.
\end{rmk}

Before proving our main result, we provide a concrete description of $L_k(\fk_\epsilon)$ for nonzero $\epsilon$, which will follow from the following simple lemma.
\begin{lem}
	\label{lem:sln}
	For $N>1$ and $k \in \Z$,  $L_k(\fsl_N)$ is a vertex subalgebra of $L_k(\fsl_{1|N})$.
\end{lem}
\begin{proof}
	The statement is proven for $k=1$ in Theorem 5.5 of \cite{CKLR19}, see also \cite{KW01}. For $k>1$ one has the well-known embedding $\iota: L_k(\fsl_N) \hookrightarrow  L_1(\fsl_{N})^{\otimes k} \hookrightarrow L_1(\fsl_{1|N})^{\otimes k}$. 
	For $x$ in $\fsl_{1|N}$ let $J^x_a$ be the corresponding field in the $a$-th  factor of $L_1(\fsl_{1|N})^{\otimes k}$, so that $J^x_1 + \dots + J^x_k$ generates a homomorphic image of $V^k(\fsl_{1|N})$. 
	Then the image of $\iota$ is generated by the $J^x_1 + \dots + J^x_k$ for $x \in \fsl_N \subset \fsl_{1|N}$ 
	and so $L_k(\fsl_N)$ is in particular a vertex subalgebra of the simple quotient $L_k(\fsl_{1|N})$.
\end{proof}

\begin{cor}
	\label{cor:osp}
	For $\epsilon \neq 0$ the quotient $L_k(\fk_{\epsilon})$ is isomorphic to $L_k(\fosp_{1|2n})$.
\end{cor}
\begin{proof}
	We start by noting that $L_k(\fk_\epsilon)$ is a homomorphic image of $V^k(\fosp_{1|2n})$ due to Corollary \ref{cor:affineosp}, hence contains a homomorphic image of $V^k(\fsp_{2n})$.
	The above Lemma ensures that the $(k+1)$st powers of the nilpotent generators of this subalgebra vanish and it is therefore isomorphic to $L_k(\fsp_{2n})$.
	The corollary is then an immediate consequence of Theorem 4.5.2 of \cite{GorelikSerganova}, as explained in the proof of Theorem 5.3 thereof.
\end{proof}

\subsection{The main theorem}

We now establish a relatively simple variant of the main result of \cite{Stoyanovsky1998}.

\begin{thm}
	\label{thm:characters}
	The supercharacter of $L_k(\fosp_{1|2n})$ is equal to supercharacter of principal subspace of $L_k(\fsl_{1|2n+1})$ and to the character of the principal subspace of $L_k(\fsl_{2n})$.
\end{thm}

\begin{proof}
	In light of Proposition \ref{cor:osp}, it suffices to show that Corollary \ref{cor:supercharacters} applies to this setting.
	Namely, that every element of the $E$-weight zero subspace of $I_\fp$ can be deformed to $I_\epsilon$.
	
	To show this, we first note that the $E$-weight zero subspace of $V^k(\fp)$ is precisely the universal principal subspace of $\fsl_{2n}$.
	This follows from the fact that the $E$-weight zero subspace of $\fp$ is precisely the nilpotent subsalgebra of $\fsl_{2n}$.
	Using Lemma \ref{lem:sln}, it is immediate that the $E$-weight zero subspace of $I_\fp$ is precisely the defining ideal of the principal subspace of $L_k(\fsl_{2n})$.
	The generators of this ideal are stated by Feigin-Stoyanovsky \cite{StoyanovskyFeigin} for all $\fsl_m$ (see also \cite{Sadowski1}) and proven only for $\fsl_2$; the case of $\fsl_3$ is proven by Sadowski \cite{Sadowski2} and we show that this continues to hold for all $m$ in Appendix \ref{app:presentation} following work of Butorac-Ko\v{z}i\'{c} providing an analogous presentation for types $D$, $E$, and $F$ \cite{ButoracKozic}.
	Explicitly, the ideal defining the principal subspace of $L_k(\fsl_{2n})$ is generated by the modes of the fields $e_{i}(z)^{k+1}$.
	
	To complete the proof, we show that each of these generators can be deformed away from $\epsilon = 0$.
	We claim that this is done by simply replacing $e_i(z)^k$ by $e_{i,\epsilon}(z)^{k+1}$, where $e_{i, \epsilon}(z) = e_{i}(z) - \epsilon^2 f_{i+1}(z)$ for $i = 2, \dots 2n-1$, $e_{2n,\epsilon}(z) = e_{\alpha_{2n}}(z)$.
	That these fields correspond to states in the $E$-weight zero subspace of $V^k(\fk_{\epsilon})$ is immediate and so it suffices to show that they belong to the $E$-weight zero subspace of $I$, i.e. the maximal ideal of $L_k(\fsl_{2n})$, but this follows as in \cite{Stoyanovsky1998} from the fact that $e_{i,\epsilon}$ is nilpotent.
\end{proof}

The homology of $\fp$ is isomorphic to the nilpotent subalgebra $\fn$ of $\fsl_{2n}$ and hence $H(V^k(\fp))$ is isomorphic to $V^k(\fn)$.
The last paragraph of this proof shows that the homology $H(I_\fp)$ of the ideal $I_\fp$ defining the principal subspace of $L_k(\fsl_{1|2n+1})$ is precisely the ideal defining the principal subspace of $L_k(\fsl_{2n})$.
Applying Lemma \ref{lem:SES}, we obtain an explicit description of the homology $H(L_k(\fp))$.
\begin{cor}
	The homology $H(L_k(\fp))$ of the principal subspace of $L_k(\fsl_{1|2n+1})$ is isomorphic to the principal subspace of $L_k(\fsl_{2n})$.
\end{cor}

We end by noting that Conjecture 1.1 (for $p = k$) of \cite{WarnaarZudilin} is a consequence of Theorem \ref{thm:characters}.
Namely, we recognize the right-hand side of this conjecture as the supercharacter of $L_k(\fosp_{1|2n})$, see Appendix \ref{app:osp} for more details, and the left-hand side as the character of the principal subspace of $L_k(\fsl_{2n})$, cf. \cite{StoyanovskyFeigin, Georgiev}.
Equating these $q$-series leads to the desired result.
\begin{cor}[Conj. 1.1 of \cite{WarnaarZudilin}]
	\label{cor:characters}
	For $n, k \geq 0$,
	\begin{equation}
	\label{eq:characters}
		\frac{1}{(q)_\infty^{n(2n-1)}}\sum_{u \in \Z^n}(-1)^{|u|} \xi(u) q^{(k+n+\scriptstyle{\frac{1}{2}})|\!|u|\!|^2+ \tilde\rho \cdot u} = \sum_{m \in \mathbb{N}^{k(2n-1)}} \frac{q^{\scriptstyle{\frac{1}{2}}\sum\limits_{i,j=1}^{k}\sum\limits_{a,b=1}^{2n-1} T_{ij} A_{ab} m_{ia} m_{jb}}}{\prod\limits_{i=1}^k \prod\limits_{a=1}^{2n-1} (q)_{m_{ia}}}
	\end{equation}
	where $|u|$ denotes the sum of the components of $u$, $|\!|u|\!|^2 = u \cdot u$ the squared norm of $u$, $\tilde\rho$ the vector with components $\tilde\rho_m = m - \frac{1}{2}$, and finally
	\[
		\xi(u) = \prod_{1 \leq l < m \leq n} \frac{v_m^2 - v_l^2}{\tilde\rho_m^2 - \tilde\rho_l^2} \qquad v_m = \tilde\rho_m + (2(n+k)+1)u_m\,.
	\]
\end{cor}

\appendix 

\section{Some background on \texorpdfstring{$\fosp_{1|2n}$}{osp1|2n}}
\label{app:osp}
We follow the appendix A of \cite{CGL} and use the following notation:
\begin{itemize}
	\item $\mathfrak{g}=\mathfrak{osp}_{1|2n}$,
	\item $(\cdot|\cdot)$  a non-degenerate consistent supersymmetric invariant bilinear form,
	\item $\mathfrak{h}$ a Cartan subalgebra of  $\mathfrak{osp}_{1|2n}$, 
	\item $\Delta$  the root system of $\mathfrak{osp}_{1|2n}$ with respect to $\mathfrak{h}^*$, \item $\Pi=\{\alpha_1, \ldots, \alpha_n\}$  a set of simple roots of $\Delta$,
	\item $\Delta^+_0$ and $\Delta^+_1$ the sets of positive even and odd roots.
\end{itemize}

Then the highest root of $\mathfrak{osp}_{1|2n}$ is equal to $\theta = 2\alpha_1 + \cdots + 2\alpha_n$.
We take $\alpha_n$ to be the (unique) non-isotropic odd simple root.
The bilinear form $(\cdot|\cdot)$ on $\mathfrak{osp}_{1|2n}$ is normalized as $(\theta|\theta) = 2$, and the $\alpha_i$ satisfy 
\begin{align*}
	&(\alpha_i|\alpha_i) = 1,\quad
	(\alpha_i|\alpha_{i+1}) = -\frac{1}{2},\quad
	i=1, \ldots, n-1,\\
	&(\alpha_n|\alpha_n) = \frac{1}{2},\quad
	(\alpha_i|\alpha_j) = 0,\quad
	|i-j|>1.
\end{align*}
The Weyl group is defined to be the Weyl group of the even subalgebra $\fsp_{2n}$.
We identify $\mathfrak h^*$ with $\mathfrak h$ as usual, that is via $\nu: \mathfrak h^* \rightarrow \mathfrak h$ given by $(\nu(\alpha)|h) = \alpha(h)$

\subsection{The Weyl (super)dimension formula}
The character $\chi_\lambda$ of a finite-dimensional irreducible module of $\fosp_{1|2n}$ of highest-weight $\lambda$ is given by 
\cite[Theorem 2.35]{CW-book} 
\[
	\chi_\lambda = \frac{\prod\limits_{\alpha\in \Delta^+_1} (1 + e^{-\alpha})}{\prod\limits_{\alpha\in \Delta^+_0} (1 - e^{-\alpha})} \sum_{w \in W} (-1)^{\ell(w)} e^{w(\lambda+\rho) -\rho}
\]
while the supercharacter is obtained by changing the sign for every $e^\mu$ with $\mu \notin \lambda + Q_0$ for $Q_0$ the root lattice of $\fsp_{2n}$.
A Weyl reflection satisfies $\omega(\mu) \in \alpha_n + Q_0$ for any $\mu \in \alpha_n + Q_0$.
The reflection satisfies $\omega(\rho) - \rho \in Q_0$ if $\omega$ is a Weyl reflection corresponding to a short even root; otherwise $\omega(\rho) - \rho$ is in $Q_0 + \alpha_n$.

Let $(-1)^{\widetilde \ell(w)}$ be equal to one if $w$ is a product of an even number of Weyl reflections corresponding to short roots times some number of reflections corresponding to long roots.
Otherwise, we set $(-1)^{\widetilde \ell(w)}$ to minus one.
Hence, the supercharacter $\wt{\chi}_\lambda$ for $\lambda \in Q$ and even highest-weight vector is given by
\[
	\widetilde \chi_\lambda = \frac{\prod\limits_{\alpha\in \Delta^+_1} (1 - e^{-\alpha})}{\prod\limits_{\alpha\in \Delta^+_0} (1 - e^{-\alpha})} \sum_{w \in W} (-1)^{\widetilde \ell(w)} e^{w(\lambda+\rho) -\rho}.
\]
\begin{cor}
	Let $\rho_\lambda$ be an irreducible finite-dimensional highest-weight module of $\fosp_{1|2n}$ with highest weight $\lambda$ in the root lattice $Q$ of $\fosp_{1|2n}$ and such that the highest-weight vector is even.
	Then
	\[
		\text{dim}(\rho_\lambda) = \prod_{\alpha \in \Delta_0^+} \frac{(\alpha| \lambda+\rho)}{(\alpha|\rho)}, \qquad 
		\text{sdim}(\rho_\lambda) = \prod_{\alpha \in \Delta_{0, \text{short}}^+} \frac{(\alpha| \lambda+\rho)}{(\alpha|\rho)}
	\]
	where $\Delta_{0, \text{short}}^+$ is the set of short even positive roots. 
\end{cor}
\begin{proof}
	Let $A_\mu = \sum_{w \in W} (-1)^{\ell(w)} e^{w(\mu) -\rho} $ and $\widetilde A_\mu = \sum_{w \in W} (-1)^{\widetilde \ell(w)} e^{w(\mu) -\rho} $. Since $\chi_0 = 1 = \widetilde \chi_0$ it follows that
	\[
		A_\rho =  \frac{\prod\limits_{\alpha\in \Delta^+_0} (1 - e^{-\alpha})}{\prod\limits_{\alpha\in \Delta^+_1} (1 + e^{-\alpha})} = \prod\limits_{\alpha\in \Delta^+_0} (1 - e^{-\frac{\alpha}{(\alpha|\alpha)}}), \qquad 
		\widetilde A_\rho =  \frac{\prod\limits_{\alpha\in \Delta^+_0} (1 - e^{-\alpha})}{\prod\limits_{\alpha\in \Delta^+_1} (1 - e^{-\alpha})}= \prod\limits_{\alpha\in \Delta^+_0} (1 + (-1)^{(\alpha|\alpha)} e^{-\frac{\alpha}{(\alpha|\alpha)}}).
	\]
	Thus 
	\begin{equation}
		\begin{split}
			A_{\lambda+ \rho}(t\rho^\vee) &= \sum_{w \in W} (-1)^{\ell(w)} e^{(w(\lambda+ \rho) -\rho| t \rho)} 
			= \sum_{w \in W} (-1)^{\ell(w)} e^{t(w(\lambda+ \rho) -\rho|  \rho)} \\
			&= \sum_{w \in W} (-1)^{\ell(w)} e^{t(\lambda+ \rho|  w(\rho)) - t(\rho + \lambda - \lambda|\rho)} 
			= e^{t(\lambda|\rho)} A_{\rho}(t(\lambda+\rho)^\vee)
		\end{split}
	\end{equation}
	and analogously 
	$\widetilde A_{\lambda+ \rho}(t\rho^\vee) =  e^{t(\lambda|\rho)} \widetilde A_{\rho}(t(\lambda+\rho)^\vee)$ hence
	\begin{equation}\label{eq:dims}
		\begin{split}
			\dim(\rho_\lambda) &= \lim_{t \rightarrow 0} \frac{A_{\lambda+\rho}(t\rho^\vee)}{A_\rho(t\rho^\vee)} = 
			\lim_{t \rightarrow 0} \frac{A_{\rho}(t(\lambda+\rho)^\vee)}{A_\rho(t\rho^\vee)} = \prod_{\alpha \in \Delta_0^+} \frac{(\alpha| \lambda+\rho)}{(\alpha|\rho)} \\
			\text{sdim}(\rho_\lambda) &= \lim_{t \rightarrow 0} \frac{\widetilde A_{\lambda+\rho}(t\rho^\vee)}{\widetilde A_\rho(t\rho^\vee)} = 
			\lim_{t \rightarrow 0} \frac{\widetilde A_{\rho}(t(\lambda+\rho)^\vee)}{\widetilde A_\rho(t\rho^\vee)} = \prod_{\alpha \in \Delta_{0, \text{short}}^+} \frac{(\alpha| \lambda+\rho)}{(\alpha|\rho)}
		\end{split}
	\end{equation}
\end{proof}
We note that \eqref{eq:dims} holds for any $\lambda \in Q \otimes_\Z\C$, 
\begin{equation}\label{eq:dims2}
	\begin{split}
		\lim_{t \rightarrow 0} \frac{A_{\lambda+\rho}(t\rho^\vee)}{A_\rho(t\rho^\vee)} &= \lim_{t \rightarrow 0} \frac{A_{\rho}(t(\lambda+\rho)^\vee)}{A_\rho(t\rho^\vee)} = \prod_{\alpha \in \Delta_0^+} \frac{(\alpha| \lambda+\rho)}{(\alpha|\rho)} \\
		\lim_{t \rightarrow 0} \frac{\widetilde A_{\lambda+\rho}(t\rho^\vee)}{\widetilde A_\rho(t\rho^\vee)} &= \lim_{t \rightarrow 0} \frac{\widetilde A_{\rho}(t(\lambda+\rho)^\vee)}{\widetilde A_\rho(t\rho^\vee)} = \prod_{\alpha \in \Delta_{0, \text{short}}^+} \frac{(\alpha| \lambda+\rho)}{(\alpha|\rho)}
	\end{split}
\end{equation}

\subsection{The supercharacter of $L_k(\fosp_{1|2n})$ for $k \in \Z_{>0}$}

Let $\widehat{\mathfrak{g}}=\mathfrak{osp}_{1|2n}[t, t^{-1}] \oplus \C K \oplus \C D$ be the (untwisted) affine Lie superalgebra of $\mathfrak{osp}_{1|2n}$. The  Lie superbrackets are
\begin{align*}
	&[a\otimes t^{m_1}, b\otimes t^{m_2}] = [a,b]\otimes t^{m_1+m_2} + m_1(a|b)\delta_{m_1+m_2,0}K,\\
	&[D, a\otimes t^{m_1}]=m_1 a\otimes t^{m_1},\quad
	[K, \widetilde{\mathfrak{g}}]=0, \qquad a, b \in \mathfrak{osp}_{1|2n}, \ \ m_1, m_2 \in \Z.
\end{align*}
Let $\widehat{\mathfrak{h}} = \mathfrak{h} \oplus \C K\oplus\C D$ be a Cartan subalgebra of $\widehat{\mathfrak{g}}$.
The bilinear form on $\mathfrak{h}$ extends to $\widehat{\mathfrak{h}}$ via $(K|D)=1$ and $(K|K) = (D|D) = (h|K) = (h|D) =0$ for all $h\in \mathfrak{h}$. 
Let $\widehat{\Delta}$ be the root system of $\widehat{\mathfrak{g}}$ with respect to $\widehat{\mathfrak{h}}^*$ and $\widehat{\Pi} = \{\alpha_0\}\sqcup\Pi$ be a set of simple roots of $\widehat{\Delta}$.
The imaginary root is $\delta:=\alpha_0+\theta$ in $\widehat{\Delta}$ and $\Lambda_0 \in \widetilde{\mathfrak{h}}^*$ such that $\delta(D)=\Lambda_0(K)=1$ and $\delta(h)=\Lambda_0(h)=\delta(K)=\Lambda_0(D)=0$ for all $h \in \mathfrak{h}$.
We have $\widehat{\mathfrak{h}}^* = \mathfrak{h}^* \oplus \C\delta \oplus \C \Lambda_0$, $\widehat{\nu}(\delta)=K$ and $\widehat{\nu}(\Lambda_0)=D$.
Denote by $\alpha^\vee = 2\alpha/(\alpha|\alpha) \in \widehat{\mathfrak{h}}^*$ for $\alpha \in \widehat{\mathfrak{h}}^*$ if $(\alpha|\alpha) \neq 0$ (here we identify $\widehat{\mathfrak{h}}^*$ with $\widehat{\mathfrak{h}}$ as done in \cite{KWchar}).
Let $\rho$ be the Weyl vector of $\fosp_{1|2n}$, that is $(\rho|\alpha_i^\nu)=1$ for $i=1, \dots, n$ and set $\widehat\rho = \rho + (n+\frac{1}{2})\Lambda_0$, where $n+\frac{1}{2}$ is the dual Coxeter number of $\fosp_{1|2n}$. Let $\lambda = k\Lambda_0$ for $k \in \Z_{>0}$, then 
\[
	(\lambda + \widehat\rho)  ((\alpha + m\delta)^\vee) = (\lambda + \widehat \rho)( \alpha^\vee + \frac{2mK}{\alpha^2}) = \frac{2m}{\alpha^2]} (k+ n +\frac{1}{2}) + \rho(\alpha^\vee).
\]
Thus 
\[
	\Delta^\lambda_{1} = \{ \alpha \in \widehat \Delta_1 | (\lambda + \widehat\rho)(\alpha^\vee) \in \Z_{\text{odd}}\} = \{ \alpha^\vee + m\delta | \alpha \in \Delta_1,  m \in 2\Z\} 
\]
and set 
\[
	\Delta^\lambda_{0} = \{ \alpha \in \widehat \Delta_0, \frac{\alpha}{2} \notin \Delta_1 | (\lambda + \widehat\rho)(\alpha^\vee) \in \Z\} \cup \frac{1}{2}\Delta^\lambda_0, \qquad \Delta^\lambda = \Delta^\lambda_{0} \cup \Delta^\lambda_{1}
\]
Then the Weyl group $W^\lambda$ is generated by all the reflections $r_{\alpha^\vee}$ for $\alpha$ in $\Delta^\lambda$.
We observe that this group is the semi-direct product of the finite Weyl group $W$ with the translations 
\begin{align*}
	t_\alpha(\lambda):=\lambda+\lambda(K)\alpha-((\lambda|\alpha)+\frac{1}{2}(\alpha|\alpha)\lambda(K))\delta,\qquad \text{for} \ \alpha \in Q^\vee.
\end{align*}
The odd positive roots are $\epsilon_n:= \alpha_n, \epsilon_{n-1} := \alpha_{n-1} +\epsilon_n, \dots, \epsilon_1 := \alpha_1 + \epsilon_2$.
They satisfy $(\epsilon_i| \epsilon_j) = \frac{1}{2}\delta_i \delta_j$, the set of short even roots is $\{ \pm \epsilon_i \pm \epsilon_j| i \neq j \}$ and the long ones are $2\epsilon_i$ for $i=1, \dots, n$; and so  $Q\cong\frac{1}{\sqrt{2}} \Z^n$.
Similarly $Q^\vee \cong \sqrt{2}\Z^n$ is generated by $2\epsilon_1, \dots, 2\epsilon_n$.
The Weyl vector is 
\[
	\rho = \frac{1}{2}\left(\epsilon_n + 3 \epsilon_{n-1} + \dots + (2n-1)\epsilon_1 \right)
\]
By  \cite[Theorem $1^s$]{KWchar} the character and supercharacter of $L_k(\fosp_{1|2n})$ satisfies
\begin{equation}
	\begin{split}
		\text{ch}[L_k(\fosp_{1|2n})] &= \sum_{w \in W^\lambda}(-1)^{\ell(w)}\text{ch}[M(w.\lambda)]     \\
		\text{sch}[L_k(\fosp_{1|2n})] &= \sum_{w \in W^\lambda}(-1)^{\widetilde\ell(w)}\text{sch}[M(w.\lambda)] 
	\end{split}
\end{equation}
with $M(\mu)$ the Verma module of highest-weight $\mu$ and $w.\lambda = w(\lambda+\rho)-\rho$
The character and supercharacter of $M(\mu)$ is 
\[
	\text{ch}[M(\mu)] =  \frac{q^{h_\mu - \frac{c_k}{24}} D^+}{A_\rho }, \qquad\
	\text{sch}[M(\mu)] =  \frac{(-1)^{|v_\mu|} q^{h_\mu - \frac{c_k}{24}} D^-}{\widetilde A_\rho }, \qquad D^\pm := \frac{\prod\limits_{\substack{n=1 \\ \alpha \in \Delta_1}}^\infty (1 \pm e^{-\alpha}q^n)}{ \prod\limits_{\substack{n=1 \\ \alpha \in \Delta_0}}^\infty (1 - e^{-\alpha}q^n)}
\]
with $v_\mu$ the parity of the highest-weight vector and
\[
	h_\mu = \frac{(\mu|\mu+2\rho)}{2\kappa}, \qquad c_k = \frac{kn(2n-1)}{2\kappa}, \qquad \kappa = k+n+\frac{1}{2}. 
\]
Thus 
\begin{equation}
	\begin{split}
		\text{ch}[L_k(\fosp_{1|2n})] &= \frac{D^+}{A_\rho} q^{-\frac{c_k}{24}}\sum_{w \in W, \lambda \in \kappa Q^\vee}(-1)^{\ell(w)} e^{w(\lambda+\rho) -\rho}q^{h_\lambda} \\     
		&= D^+ q^{-\frac{c_k}{24}}\sum_{\lambda \in \kappa Q^\vee} \frac{A_{\lambda+\rho}}{A_\rho} q^{h_\lambda} \\ 
		\text{sch}[L_k(\fosp_{1|2n})] &=  \frac{D^-}{\widetilde A_\rho} q^{-\frac{c_k}{24}}\sum_{w \in W, \lambda \in \kappa Q^\vee}(-1)^{\widetilde\ell(w) + 2\lambda^2} e^{w(\lambda+\rho) -\rho}q^{h_\lambda}     \\ 
		&= D^- q^{-\frac{c_k}{24}}\sum_{\lambda \in \kappa Q^\vee} (-1)^{2\lambda^2}\frac{\widetilde A_{\lambda+\rho}}{\widetilde A_\rho} q^{h_\lambda} 
	\end{split}
\end{equation}
The claimed character formula follows using \eqref{eq:dims2}, that is the left-hand side of \eqref{eq:characters} is the specialization of $\text{sch}[L_k(\fosp_{1|2n})]$, 
\begin{equation}
	\lim_{t\rightarrow 0}  \text{sch}[L_k(\fosp_{1|2n})](t\rho^\vee) =
	\frac{1}{(q)_\infty^{n(2n-1)}}\sum_{u \in \Z^n}(-1)^{|u|} \xi(u) q^{(k+n+\scriptstyle{\frac{1}{2}})|\!|u|\!|^2+ \tilde\rho \cdot u} 
\end{equation}    
where $|u|$ denotes the sum of the components of $u$, $|\!|u|\!|^2 = u \cdot u$ the squared norm of $u$, $\tilde\rho$ the vector with components $\tilde\rho_m = m - \frac{1}{2}$, and finally
\[
	\xi(u) = \prod_{1 \leq l < m \leq n} \frac{v_m^2 - v_l^2}{\tilde\rho_m^2 - \tilde\rho_l^2} \qquad v_m = \tilde\rho_m + (2(n+k)+1)u_m\,.
\]

\section{Presentation for the principal subspace of \texorpdfstring{$L_k(\fsl_{N+1})$}{LkslN+1}}
\label{app:presentation}
In this appendix we provide a presentation of the principal subspace of $L_k(\fsl_{N+1})$, following the work of Georgiev \cite{Georgiev} and of Butorac-Ko{\v z}i{\' c} \cite{ButoracKozic}.

We start with some notation, which largely overlaps with that of the main text.
Let $\fg = \fsl_{N+1}$ and denote by $\widehat{\fg}$ its affinization.
For $X \in \fg$ denote by $X_m = X \otimes t^m \in \widehat{\fg}$.
Choose a triangular decomposition $\fg = \fn_+ \oplus \fh \oplus \fn_-$; we set $\bar{\fn}_\pm = \fn_\pm \otimes \C[t]$.
Let $\Delta$ denote the set of roots with respect to $\fh$, $\Delta_\pm$ the set of positive/negative roots, and $\Pi = \{\alpha_1, \dots, \alpha_N\}$ the set of positive simple roots.
We will make use of the Chevalley basis $\{x_\alpha\}_{\alpha \in \Delta} \cup \{h_{\alpha_i}\}_{i=1}^N$.

We denote by $v_k$ the highest weight vector of the simple vacuum module $L_k(\fsl_{N+1})$ at level $k$.
As defined by Feigin-Stoyanovsky \cite{StoyanovskyFeigin}, the principal subspace of $L_k(\fsl_{N+1})$ is
\[
	W_{L_k(\fsl_{N+1})} := U(\bar{\fn}_+)v_k
\]
where $U(\ff)$ denotes the universal enveloping algebra of a Lie algebra $\ff$.
The principal subspace is clearly a quotient of $U(\bar{\fn}_+)$; the aim of this appendix is to show that the left ideal $\mathcal{I}_k$ annihilating $v_k$ is given by
\[
	\mathcal{I}_k = U(\bar{\fn}_+) \bar{\fn}^{\geq 0}_+ + \sum_{i=1}^N \sum_{m \geq k+1} U(\bar{\fn}_+) R_i(-m)\,,
\]
where $\bar{\fn}^{\geq 0}_+ = \fn \otimes \C[t]$ and 
\[
	R_i(-m) = \sum_{\substack{m_1, \dots, m_{k+1} \leq -1\\ m_1 + \dots + m_{k+1} = -m}} x_{\alpha_i, m_1} \dots x_{\alpha_i, m_{k+1}}\,.
\]
This gives us the desired presentation:
\begin{thm}
	\label{thm:presentation}
	For all positive integer integers we have
	\[
		W_{L_k(\fsl_{N+1})} \simeq U(\bar{\fn}_+) / \mathcal{I}_k\,.
	\]
\end{thm}

\subsection{Quasiparticle basis of $W_{L_k(\fsl_{N+1})}$}
We now introduce a basis of $W_{L_k(\fsl_{N+1})}$ due to \cite{Georgiev}.
We first consider the formal series
\[
	x_{\alpha_i}(z) = \sum_{m} x_{\alpha_i, m} z^{-m-1}\,, \qquad i = 1, \dots, N\,,
\]
where are fields on $L_k(\fsl_{N+1})$.
Note that $[x_{\alpha_i}(z_1), x_{\alpha_i}(z_2)] = 0$, so that
\[
	x_{n \alpha_i}(z) = x_{\alpha_i}(z)^n = \sum_{m} x_{n \alpha_i, m} z^{-m-n}
\]
is also a well-defined field on $L_k(\fsl_{N+1})$.
Note that the $R_i(-m)$ appearing in the definition of $\mathcal{I}_k$ are coefficients of $x_{(k+1) \alpha_i}(z)$.
We also note that by the Poincar\'{e}-Birkoff-Witt theorem for $U(\bar{\fn}_+)$ we have the following isomorphism of vector spaces:
\[
	U(\bar{\fn}_+) = U(\bar{\fn}_{\alpha_N}) \dots U(\bar{\fn}_{\alpha_1})
\]
where $\bar{\fn}_{\alpha_i} = \fn_{\alpha_i} \otimes \C[t^{\pm 1}]$ and $\fn_{\alpha_i} = \C x_{\alpha_i}$ is the 1-dimensional abelian Lie algebra generated by the root vector $x_{\alpha_i}$.

As in \cite{Georgiev}, we call the coefficient $x_{n \alpha_i, m}$ a quasiparticle of color $i$, charge $n$, and energy $-m$.
We also call an endomorphism of the form
\[
	b = \big(x_{n_{r_N,N}\alpha_N, m_{r_N, N}} \dots x_{n_{1,N}\alpha_N, m_{1, N}}\big) \dots \big(x_{n_{r_1,1}\alpha_1, m_{r_1, 1}} \dots x_{n_{1,1}\alpha_1, m_{1, 1}}\big)
\]
a quasiparticle monomial; we say that such a quasiparticle monomial has color-charge-type
\[
	(n_{r_N,N}, \dots, n_{1,N}; \dots; n_{r_1,1}, \dots, n_{1,1})
\]
and index sequence
\[
	(m_{r_N,N}, \dots, m_{1,N}; \dots; m_{r_1,1}, \dots, m_{1,1})\,.
\]
We call the tuples $(n_N; \dots; n_1)$ and $(m_N; \dots; m_1)$, where $n_i = n_{r_i, i} + \dots + n_{1,i}$ and $m_i = m_{r_i, i} + \dots + m_{1,i}$, its color type and its index sum.
As quasiparticles of color $i$ commute amongst themselves, we assume that $n_{a+1, i} \leq n_{a,i}$ and, if $n_{a+1, i} = n_{a,i}$, that $m_{a+1, i} \leq m_{a,i}$ without loss of generality.

We will also need two orders on the set of all quasiparticle monomials: the lexicographical (linear) order $<$ and the multidimensional (partial) order $\prec$.
For two such monomials $b, b'$ we say $b < b'$ if the color-charge-type of $b$ is less than that of $b'$ in lexicographical order; if their color-charge-types are the same, we say $b < b'$ if its index sequence is less than that of $b'$ in lexicographical order.
Additionally, we say that $b \prec b'$ if $b < b'$ and, for every $1 \leq s \leq N$, one has $m_s + \dots + m_1 \leq m_s' + \dots + m_1'$ and this is a strict inequality for at least one $s$, where $m_i = \sum_a m_{a,i}$.

Denote by $B$ the set of quasiparticle monomials satisfying
\begin{align}
	m_{a+1, i} & \leq m_{a, i} - 2 n_{a,i} \quad \text{if } n_{a+1, i} = n_{a,i} \label{eq:cond1}\\
	m_{a, i} & \leq (1-2a) n_{a,i} + \sum_{b=1}^{r_{i-1}} \min(n_{b,i-1}, n_{a,i}) \label{eq:cond2}\\
	n_{a,i} & \leq k \label{eq:cond3}
\end{align}
for all $i = 1, \dots N$ and $a = 1, \dots r_i$.

\begin{thm}[\cite{Georgiev}, Theorem 5.2]
	The set
	\[
		\mathfrak{B} = \{b v_k | b \in B\}
	\]
	is a basis of the principal subspace $W_{L_k(\fsl_{N+1})}$.
\end{thm}

\subsection{Spanning set for $U(\bar{\fn}_+)/\mathcal{I}_k$}
We now give a spanning set for the quotient $U(\bar{\fn}_+)/\mathcal{I}_k$, from which we will conclude Theorem \ref{thm:presentation}.
For $X \in U(\bar{\fn}_+)$ we denote by $\overline{X}$ its image in this quotient; we then set
\[
	\overline{\mathfrak{B}} = \{\overline{b} | b \in B\}\,.
\]

\begin{prop}
	\label{prop:spanningset}
	The set $\overline{\mathfrak{B}}$ spans the quotient $U(\bar{\fn}_+)/\mathcal{I}_k$.
\end{prop}

We will prove this proposition briefly, first using it to prove the main result of this appendix as in the proof of Theorem 4.1 of \cite{ButoracKozic}.

\begin{proof}[Proof of Theorem \ref{thm:presentation}]
	We start with the canonical surjection
	\[
		f: U(\bar{\fn}_+) \to W_{L_k(\fsl_{N+1})}\,, \quad X \mapsto X v_k\,.
	\]
	It is clear that $\mathcal{I}_k$ belongs to the kernel of $f$, so this map factors through $U(\bar{\fn}_+)/\mathcal{I}_k$.
	Denote by $\overline{f}$ the corresponding map from $U(\bar{\fn}_+)/\mathcal{I}_k$ to $W_{L_k(\fsl_{N+1})}$; this map is necessarily surjective and bijectively maps the spanning set $\overline{\mathfrak{B}}$ to the basis $\mathfrak{B}$.
	We conclude that $\overline{\mathfrak{B}}$ is a basis of $U(\bar{\fn}_+)/\mathcal{I}_k$ and that $\overline{f}$ is an isomorphism of vector spaces.
\end{proof}

We end this appendix by sketching a proof of Proposition \ref{prop:spanningset}.
The proof closely mirrors the proof of Theorem 5.1 of \cite{Georgiev} and parts of the proof of Theorem 4.1 of \cite{ButoracKozic}, so we only provide the main ideas.

\begin{proof}[Proof of Proposition \ref{prop:spanningset}]
	We note three properties satisfied by the formal series $x_{n \alpha_i}(z)$ independent of the module of $\widehat{\fg}$ they are acting on and hence can be applied to $U(\bar{\fn}_+)$.
	The first two properties ultimately stem from the expression $x_{n \alpha_i}(z) = x_{\alpha_i}(z)^n$.
	The first property takes the form
	\begin{equation}
	\label{eq:prop1}
	\begin{aligned}
		x_{n \alpha_i}(z) x_{n' \alpha_i}(z) & = x_{(n-1) \alpha_i}(z) x_{(n'+1) \alpha_i}(z)\\
		& = \dots = x_{\alpha_i}(z) x_{(n+n'-1)\alpha_i}(z) = x_{(n+n')\alpha_i}(z)
	\end{aligned}
	\end{equation}
	for $0 < n \leq n'$, cf. Eq. (3.18) of \cite{Georgiev}; this induces $n$ relations on quasiparticle monomials of color $i$.
	The second property we will need is given by
	\begin{equation}
	\label{eq:prop2}
	\begin{aligned}
		\frac{n + n'}{n}\bigg(\frac{d \hfill}{d z}x_{n \alpha_i}(z)\bigg) x_{n' \alpha_i}(z) & = \frac{n + n'}{n-1}\bigg(\frac{d \hfill}{d z}x_{(n-1) \alpha_i}(z)\bigg) x_{(n'+1) \alpha_i}(z)\\
		& = \dots = \frac{n + n'}{1}\bigg(\frac{d \hfill}{d z}x_{\alpha_i}(z)\bigg) x_{(n+n'-1) \alpha_i}(z)\\
		& = \frac{d \hfill}{d z}\bigg(x_{n \alpha_i}(z) x_{n' \alpha_i}(z)\bigg)
	\end{aligned}
	\end{equation}
	for $0 < n \leq n'$, cf. Eq. (3.22) of \cite{Georgiev}; this induces another $n$ relations on quasiparticle monomials of color $i$.
	
	The third and final property we will need is as follows.
	We consider the generating function of quasiparticle monomials of a fixed color-charge-type
	\[
		X_{\{n_{a,i}\}}(\{z_{a,i}\}) = x_{n_{r_N,N}}(z_{r_N,N}) \dots x_{n_{1,N}}(z_{1,N}) \dots x_{n_{r_1,1}}(z_{r_1,1}) \dots x_{n_{1,1}}(z_{1,1})
	\]
	and multiply it by the Laurent polynomial
	\[
		P_{\{n_{a,i}\}}(\{z_{a,i}\}) = \prod_{i=2}^N \prod_{a=1}^{r_i} \prod_{b=1}^{r_{i-1}} \bigg(1-\frac{z_{b,i-1}}{z_{a,i}}\bigg)^{\min(n_{b,i-1},n_{a,i})}\,.
	\]
	The third property then says that their product, modulo $\mathcal{I}_k$, belongs to $(U(\bar{\fn}_+)/\mathcal{I}_k)[[\{z_{a,i}\}]]$:
	\begin{equation}
	\label{eq:prop3}
		P_{\{n_{a,i}\}}(\{z_{a,i}\}) X_{\{n_{a,i}\}}(\{z_{a,i}\}) + \mathcal{I}_k \in \big(U(\bar{\fn}_+)/\mathcal{I}_k\big)[[\{z_{a,i}\}]]
	\end{equation}
	cf. Lemma 5.1 of \cite{Georgiev}.
	
	As in the proof of Theorem 5.1 of \cite{Georgiev}, the properties in Eq. \eqref{eq:prop1} imply that if the color-charge-type of the quasiparticle monomial $b$ has $n_{a,i} = n_{a+1,i}$ for some $i, a$ but its index sequence does not satisfy
	\[
		m_{a+1,i} \leq m_{a, i} - 2 n_{a,i}
	\]
	then it can be re-expressed as a linear combination of quasiparticle monomials $b'$ with $b \prec b'$ of same color type and index sum, but possibly different color-charge-type; only a finite number of the $b'$ do not belong to $U(\bar{\fn}_+)\bar{\fn}_+^{\geq 0} \subset \mathcal{I}_k$.
	Thus, the first condition Eq. \eqref{eq:cond1} defining the set $B$ of quasiparticle monomials follows from the property in Eq. \eqref{eq:prop1}.
	
	If the quasiparticle monomial $b$ does not satisfy the second condition Eq. \eqref{eq:cond2} defining $B$
	\[
		m_{a,i} \leq (1-2a) n_{a,i} + \sum_{b=1}^{r_{i-1}} \min(n_{b,i-1}, n_{a,i})
	\]
	then the properties in Eqs. \eqref{eq:prop1}, \eqref{eq:prop2}, \eqref{eq:prop3} imply that it can be re-expressed as a linear combination of quasiparticle monomials $b' \succ b$ of the same color type and total index sum, only a finite number of which do not belong to $U(\bar{\fn}_+)\bar{\fn}_+^{\geq 0} \subset \mathcal{I}_k$.
	This is proven exactly as in Theorem 5.1 of \cite{Georgiev} and follows by induction on the color type and index sum.
	
	For the last condition Eq. \eqref{eq:cond3} defining the set $B$, we proceed as in the last paragraph in Section 5.5 of \cite{ButoracKozic}.
	If the color-charge-type of quasiparticle monomial $b$ satisfies the first two conditions Eqs. \eqref{eq:cond1}, \eqref{eq:cond2} and has $n_{a,i} > k$ for some $a,i$, i.e. it does not satisfy Eq. \eqref{eq:cond3}, then the commutation relations of the $x_{n \alpha_j}$ imply that we can bring the term $x_{n_{a,i}}(z)$ in the product $P_{\{n_{a,i}\}}(\{z_{a,i}\}) X_{\{n_{a,i}\}}(\{z_{a,i}\})$ all the way to the right and hence we find that this product belongs to $\mathcal{I}_k$.
	As $b$ is the coefficient of $z_{r_N, N}^{m_{r_N,N}+n_{r_N, N}} ... z_{1, 1}^{m_{1,1}+n_{1, 1}}$ of $X_{\{n_{a,i}\}}(\{z_{a,i}\})$, we can take this same coefficient of $P_{\{n_{a,i}\}}(\{z_{a,i}\}) X_{\{n_{a,i}\}}(\{z_{a,i}\})$ to see that $b$ can be expressed as a linear combination of quasiparticle monomials $b'$ of the same color-charge-type.
	Only a finite number of such $b'$ do not belong to $\mathcal{I}_k$, so we can continue this process to ultimately show that $b$ itself belongs to $\mathcal{I}_k$.
\end{proof}

\bibliography{deformations}
\bibliographystyle{amsalpha}
	
\end{document}